\documentclass[11pt,a4paper]{amsart}

\usepackage{amscd,amssymb,mathrsfs}
\usepackage{mathtools}
\usepackage{enumerate}
\usepackage{float}
\usepackage{tabularx}
\newcolumntype{C}{>{\centering\arraybackslash}X}

\usepackage{xcolor}
\usepackage{tikz-cd}
\usepackage{tikz}

\usepackage[hidelinks]{hyperref}
\usepackage{orcidlink}

\newcommand{\Q}{\mathbb Q}
\newcommand{\R}{\mathbb R}

\newcommand{\htilde}{\tilde{h}}

\newcommand\im{\operatorname{im}}

\newcommand\id{\operatorname{id}}

\newtheorem{theorem}{Theorem}
\newtheorem{lemma}[theorem]{Lemma}
\newtheorem{proposition}[theorem]{Proposition}
\newtheorem{corollary}[theorem]{Corollary}

\theoremstyle{definition}

\theoremstyle{remark}
\newtheorem{remark}[theorem]{Remark}

\begin{document}

\title{Handlebodies separate elements of surface groups}

\author[Richard Hain]{Richard Hain\, \orcidlink{0000-0002-7009-6971}}
\address{Department of Mathematics\\ Duke University\\ Durham, NC 27708-0320}
\email{hain@math.duke.edu}

\thanks{The author gratefully acknowledges generous support of a Simons Foundation travel grant}

\date{\today}



\maketitle

\section{Introduction}

The goal of this short note is to prove the following statement.

\begin{proposition}
\label{prop:main}
If $X$ is a closed oriented surface  of genus $g \ge 1$ and $p\in X$, then
$$
\bigcap_{X = \partial U} \ker\{\pi_1(X,p) \to \pi_1(U,p)\}
$$
is trivial, where $U$ ranges over all handlebodies with boundary $X$.
\end{proposition}

The proof is a refinement of the proof of a similar result \cite[Lem.~2.2]{leininger-reid} of Leininger and Reid. Its relevance was pointed out to the author by Sebastian Hensel. The only new ingredient is Lemma~\ref{lem:lift} and its Corollary~\ref{cor:F_g}.

The result is used in \cite{dupont-hain} to establish a lower bound on the action of the pure braid group $P_m(X)$ of an $m$-punctured surface $Y=X-\{p_1,\dots,p_m\}$ on the to cohomology group $H^n(F_n(Y);\Q)$ of the ordered configuration space $F_n(Y)$ of $Y$.
\bigskip

\noindent{\em Acknowledgments:} I am grateful to Sebastian Hensel for suggesting that the proposition should follow from the result \cite[Lem.~2.2]{leininger-reid} of Leininger and Reid and the fact, due to Baumslag \cite{baumslag}, that surface groups are residually free.

\medskip

\noindent{\em AI statement:} No large language models were used to suggest proofs, generate examples, or to aid in the writing of this manuscript.

\section{Preliminaries}

Suppose that $\gamma$ is a non-trivial element of $\pi_1(X,p)$. We have to show that there is a handlebody $U$ with boundary $X$ such that the image of $\gamma$ in $\pi_1(U,p)$ is non-trivial. Since surface groups are residually free \cite{baumslag}, there is a free group $F$ and a homomorphism $\phi : \pi_1(X,p) \to F$ such that $\phi(\gamma)\neq 1$. Since subgroups of free groups are free, we can (and will) assume that $\phi$ is surjective. In this case,
$$
\phi^\ast : H^1(F) \to H^1(X)
$$
is injective and has primitive image. (That is, $H^1(X)/\im \phi^\ast$ is torsion free.) Since $\phi^\ast$ is a ring homomorphism, $\im \phi^\ast$ is an isotropic subspace of $H^1(X)$. It follows that the rank $r$ of $F$ satisfies $r\le g$.

We can realize $\phi$ by a continuous mapping $f : X \to W$, where $W$ is a bouquet of $r$ circles, where $f_\ast(\gamma)$ is non-trivial in $\pi_1(W,\ast)$. We can assume that $W$ is a closed subspace of $\R^2$ where each each edge is smoothly embedded and the tangent directions of the edges at the ``cone point'' $\ast$ of $W$ are distinct.

\begin{lemma}
\label{lem:realization}
There is a smooth map $\Phi : X \to \R^2$ that is homotopic to $f$ and whose image is contained in $W$.
\end{lemma}

\begin{proof}[Sketch of proof]
By standard arguments, there is neighbourhood $V$ of $W$ in $\R^2$ and a smooth map\footnote{Remind folks what this means.} $r : V \to W$ satisfying:
\begin{enumerate}

\item $r(W) \subseteq W$,

\item $r : V \to W$ is a homotopy equivalence whose restriction to $W$ is homotopic to the identity $W \to W$,

\end{enumerate}
Let $\iota : W \to V$ be the inclusion. There is an $\epsilon > 0$ such that for all $w\in W$, the closed disk of radius $\epsilon$ centered at $w$ is contained in $V$. There is a smooth map $h: X \to \R^2$ for which $\|f(x)-h(x)\| < \epsilon$ for all $x\in X$. Its image is contained in $V$ and $f$ and $h$ are homotopic as maps $X \to V$ via a linear homotopy. Let $\Phi=r\circ h$.  Since $r\circ j$ is homotopic to the identity $W \to W$, we have
$$
\Phi \simeq r\circ \Phi \simeq r\circ f = r\circ \iota \circ f \simeq f
$$
where $\simeq$ denotes homotopy of maps $X \to V$.
\end{proof}

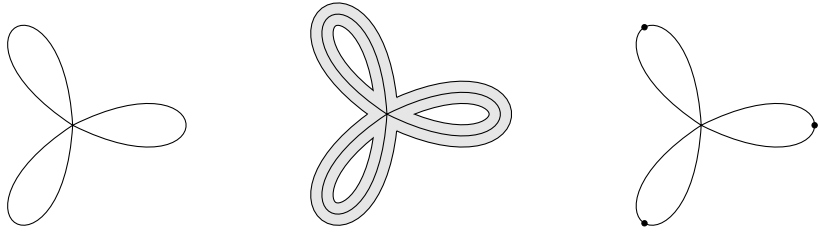
\begin{figure}[!ht]
\begin{center}
\begin{tikzpicture}[scale=0.5]
\draw (0,0) .. controls (4,-2) and (4,2) .. (0,0);
\draw [rotate=120] (0,0) .. controls (4,-2) and (4,2) .. (0,0);
\draw [rotate=240] (0,0) .. controls (4,-2) and (4,2) .. (0,0);
\end{tikzpicture}
\begin{tikzpicture}[scale=0.5]
\draw [double=black!10, double distance=8pt] (0,0) .. controls (4,-2) and (4,2) .. (0,0);
\draw [double=black!10, double distance=8pt,rotate=120] (0,0) .. controls (4,-2) and (4,2) .. (0,0);
\draw [double=black!10, double distance=8pt,rotate=240] (0,0) .. controls (4,-2) and (4,2) .. (0,0);
\filldraw[black!10] (0,0) circle [radius=14pt];
\draw (0,0) .. controls (4,-2) and (4,2) .. (0,0);
\draw [rotate=120] (0,0) .. controls (4,-2) and (4,2) .. (0,0);
\draw [rotate=240] (0,0) .. controls (4,-2) and (4,2) .. (0,0);
\end{tikzpicture}
\begin{tikzpicture}[scale=0.5]
\draw (0,0) .. controls (4,-2) and (4,2) .. (0,0);
\filldraw[black] (3,0) circle [radius=2pt];
\draw [rotate=120] (0,0) .. controls (4,-2) and (4,2) .. (0,0);
\filldraw[black] [rotate=120] (3,0) circle [radius=2pt];
\draw [rotate=240] (0,0) .. controls (4,-2) and (4,2) .. (0,0);
\filldraw[black] [rotate=240] (3,0) circle [radius=2pt];
\end{tikzpicture}
\end{center}
\caption{$W$, $(V,W)$ and $(W,P)$}
\label{fig:bouquet}
\end{figure}

The inverse image $\Phi^{-1}(e)$ of each open 1-cell $e$ of $W$ is an open subset of $X$. Applying Sard's Theorem to the restriction of $\Phi$ to $\Phi^{-1}(e)$, we see that there is a regular value of $\Phi$ on $e$. Fix one regular value $p_e \in e$ for each edge. Set $P=\{p_e\}$. The inverse image of $P$ is a finite set of simple closed curves on $X$. Removing these, we obtain a decomposition
$$
X^o := X - \Phi^{-1}(P) = \coprod_j Z_j^o
$$
where each $Z_j^o$ is a connected subsurface. Compactify each $Z_j^o$ by adding a circle to each of its ends to obtain surfaces $Z_j$ with boundary with $Z_j^o = Z_j - \partial Z_j$. Denote the genus of a compact oriented surface $Z$ by $g(Z)$.

\begin{lemma}
\label{lem:genera}
The homology classes of the the components of $\Phi^{-1}(P)$ span a primitive isotropic subspace $B$ of $H_1(X)$ of rank $r$. Moreover
$$
r + \sum_j g(Z_j) = g(X).
$$
\end{lemma}

\begin{proof}
The first statement is clear. Choose a splitting of the surjection $H_1(Z_j) \to H_1(Z_j)/\text{radical}$ for each $j$. Denote its image by $K_j$. Then $K_j$ has rank $2g(Z_j)$ and the restriction of the intersection pairing to it is unimodular. Then $H_1(X)$ decomposes
$$
\textstyle H_1(X) = A \oplus \bigoplus_j K_j \oplus B,
$$
where $A$ is orthogonal to $\bigoplus_j K_j$. The intersection pairing induces an isomorphism $A \to B^\vee$ so that $A\oplus B$ is a subgroup of $H_1(X)$ of rank $2r$ on which the intersection pairing is unimodular. This establishes the identity.
\end{proof}

\begin{remark}
This is an integral splitting of the monodromy weight filtration of the nilpotent endomorphism $\tau-\id$ of $H_1(X)$, where $\tau$ is the product of the (positive) Dehn twists about all of the simple closed curves in $f^{-1}(P)$. See \cite[Ex.~7.4]{hain:rel_wt}.
\end{remark}

\begin{lemma}
\label{lem:lift}
If $r<g(X)$, then there is a free group $F'$ of rank $r+1$, a surjective homomorphism $\phi' : \pi_1(X,p) \to F'$ and a homomorphism $\psi : F' \to F$ such that the diagram
$$
\begin{tikzcd}
\pi_1(X,p) \ar[r,twoheadrightarrow,"\phi'"] \ar[rd,twoheadrightarrow,"\phi"'] & F' \ar[d,"\psi"] \\
& F
\end{tikzcd}
$$
commutes.
\end{lemma}

\begin{proof}
Lemma~\ref{lem:genera} implies that there is a component $Z_j^o$ of $\Phi^{-1}(W-P)$ of positive genus. It contains a separating simple closed curve $C$ that divides $Z_j$ into a subsurface $E$ of genus 1 with boundary $C$ and a subsurface of of genus $g(Z_j)-1$ that contains $\partial Z_j$. We may assume that $p\in C$. Since the image of $Z_j^o$ lies in the contractible set $W-P$, the homomorphism $\pi_1(E,p) \to \pi_1(W,f(p))$ is trivial. Let $\phi_E : \phi_E : (E,p) \to (S^1,\ast)$ be a continuous map that induces a surjection $\pi_1(E,p) \to \pi_1(S^1,\ast)$. Set $F' = F\ast \pi_1(S^1)$ and let $\psi : F' \to F$ be the homomorphism with kernel generated by $\pi_1(E,p)$ whose restriction to $F$ is the identity. By Van Kampen
$$
\pi_1(X,p) \cong \pi_1(X-E^o,p) \ast_{\pi_1(C,p)} \pi_1(E,p),
$$
where $E^o = E-C$. The homomorphisms $\pi_1(X-E^o,p) \to \pi_1(X,p) \to F$ and $\psi$ induce a surjective homomorphism $\pi_1(X,p) \to F'$ whose composition with $\phi$ is $\phi$.
\end{proof}

\begin{corollary}
\label{cor:F_g}
If $\gamma$ is a non-trivial element of $\pi_1(X,p)$, then there is a free group $F$ of rank $g$ and a surjective homomorphism $\phi : \pi_1(X,p) \to F$ such that $\phi(\gamma) \neq 1$.
\end{corollary}

\section{The proof}

It suffices to show that, for each non-trivial element $\gamma\in \pi_1(X,p)$, there is a handlebody $U$ with boundary $X$ such that $\gamma$ is non-trivial in $\pi_1(U,p)$. Fix such a $\gamma$. By Corollary~\ref{cor:F_g} there is a free group $F$ of rank $g$ and a surjective homomorphism $\phi : \pi_1(X,p) \to F$ with $\phi(\gamma) \neq 1$. Realize this by a smooth map $f : X \to W$, where $W$ is a bouquet of $g$ circles as in Lemma~\ref{lem:realization}. Fix a set $P$ of regular values and decompose $X$ as described above. Lemma~\ref{lem:genera} implies that each $Z_j$ has genus 0. For each $j$ there is a 3-ball $B_j$ with $Z_j\subset \partial B_j$ and $\partial B - Z_j^o$ a union $E_j$ of closed disks.

Since $\pi_2(W,P)$ is trivial, the map $h_j : (Z_j,\partial Z_j) \to (W,P)$ extends to a continuous map $\htilde_j : (B_j,E_j) \to (W,P)$. Let
$$
U = (\coprod_j B_j\big)/\sim
$$
where $\sim$ is the equivalence relation obtained by identifying the pairs of components of $\coprod_j E_j$ that share a common boundary in $X$. This is a handlebody of genus $g$ with boundary $X$. The maps $\htilde_j$ induce a continuous map $U \to W$. The induced homomorphism $\pi_1(U,p) \to \pi_1(W,\ast)$ is clearly surjective. It is an isomorphism as both groups are free of rank $g$ and as free groups are Hopfian.

\end{document}